\documentclass[11pt]{amsart}

\usepackage[margin=1.2in]{geometry}
\usepackage[T1]{fontenc}
\usepackage{lmodern}
\usepackage{amsmath,amssymb,amsthm,mathtools}
\usepackage{xcolor}
\definecolor{linkblue}{rgb}{0.1,0.2,0.5}
\usepackage[colorlinks=true,linkcolor=linkblue,citecolor=linkblue,urlcolor=linkblue]{hyperref}

\newtheorem{theorem}{Theorem}[section]
\newtheorem{proposition}[theorem]{Proposition}
\newtheorem{lemma}[theorem]{Lemma}
\newtheorem{corollary}[theorem]{Corollary}
\newtheorem{conjecture}[theorem]{Conjecture}

\theoremstyle{definition}
\newtheorem{definition}[theorem]{Definition}
\newtheorem{example}[theorem]{Example}

\theoremstyle{remark}
\newtheorem{remark}[theorem]{Remark}

\newcommand{\Aut}{\operatorname{Aut}}

\newcommand{\SB}[2]{\mathrm{SB}(#1,#2)}
\newcommand{\RL}{\mathfrak{R}_{L}}
\newcommand{\Id}{\operatorname{Id}}

\title[The left-nilpotent residual]{The Left-Nilpotent Residual and Its Supplements\\\ in Finite Skew Braces}
\author{G\"{U}L\.{I}N ERCAN}
\address{Department of Mathematics, Middle East Technical University, Ankara, Turkey}
\email{ercan@metu.edu.tr}

\author{\.{I}SMA\.{I}L
 \c{S}.\ G\"ulo\u{g}lu}
\address{Department of Mathematics, Do\u{g}u\c{s} University, \.{I}stanbul, Turkey}
\email{iguloglu@dogus.edu.tr}

\date{August 27, 2026}

\subjclass[2020]{Primary 16T25; Secondary 20D15, 20F18, 81R50}
\keywords{skew brace, left nilpotency, left-nilpotent residual, final commutator, lambda action, ideal, nilpotent residual, Yang--Baxter equation}

\begin{document}

\begin{abstract}
Let $X$ be a finite skew brace, and let $L_\infty(X)$ be the final term
of its left series. We introduce the left-nilpotent residual $\RL(X)$
and prove that
$
\RL(X)=\Id_X\bigl(L_\infty(X)\bigr).
$
Thus $X/\RL(X)$ is the largest left-nilpotent quotient of $X$. For skew
braces of nilpotent type, we relate $\RL(X)$ to
$\gamma_\infty(X,\cdot)$ and obtain conditions ensuring that
$L_\infty(X)$ is an ideal, including the case where $|X|$ is cube-free.
Our main results concern supplements to $\RL(X)$. We prove existence
results using coprime action and Sylow and Hall theory, and give an
example of order $18$ whose residual has no proper supplement. Finally,
iterating $\RL$ yields the largest perfect subskew brace of $X$, and $X$
is KSV-solvable if and only if this subskew brace is zero.
\end{abstract}

\maketitle

\section{Introduction}\label{sec:intro}

Skew braces, introduced by Guarnieri and Vendramin \cite{GV} as a
non-commutative generalization of Rump's braces \cite{Rump}, provide an
algebraic framework for bijective non-degenerate set-theoretic
solutions of the Yang--Baxter equation.  Several notions of nilpotency
occur naturally in this setting.  In this paper we are concerned with
\emph{left nilpotency}, defined by the left series
\[
X^1=X,\qquad X^{n+1}=X*X^n.
\]
For skew braces of nilpotent type, Ced\'o, Smoktunowicz and Vendramin
proved that a finite skew brace is left nilpotent if and only if its
multiplicative group is nilpotent \cite{CSV}.

The star product is closely related to commutators arising from the
lambda action (see \cite[Section~2.3]{DelCorso}). For background on the
left series, we refer to
\cite[Chapter~12]{CedoVendraminBook}. In Section~3, we recall the
precise relation using our commutator convention and apply the
repeated- and final-commutator theory of
Isaacs--Meierfrankenfeld~\cite{IM}.
Let
\[
G=(X,+),\qquad A=(X,\cdot),\qquad
\Lambda=\lambda(A)\leq\Aut(G).
\]
Commutators with $A$ are taken through the lambda action. Since
$\ker\lambda$ acts trivially on $G$ and $\lambda(A)=\Lambda$, repeated
commutators with $A$ and with $\Lambda$ coincide. Therefore
\[
X^{n+1}
=[G,\underbrace{\Lambda,\ldots,\Lambda}_{n}]
=[G,{}_nA].
\]
Thus, for finite $X$, the final term $L_\infty(X)$ is the final
commutator of $G$ by $A$.

The first part of the paper concerns the properties of this final
term. In particular, $L_\infty(X)$ is normal in $(X,+)$
if and only if $\gamma_\infty(X^2,+)\leq L_\infty(X)$. Moreover, when
$L_\infty(X)$ is additively normal, it follows from
\cite[Proposition~2.2]{DE} that $L_\infty(X)$ is an ideal if and only
if $L_\infty(X)*X\subseteq L_\infty(X)$.
Examples of orders $18$ and $36$ show respectively that the final term
need not be an ideal and need not even be normal in the additive
group.

Since $L_\infty(X)$ need not be an ideal, it does not always define a
quotient skew brace.  We therefore introduce the \emph{left-nilpotent
residual}
\[
\RL(X)=\bigcap\{I\trianglelefteq X:X/I\text{ is left nilpotent}\}.
\]
To the best of our knowledge, this residual has not previously been
considered in the theory of skew braces.  We establish the fundamental
description
\[
\RL(X)=\Id_X\bigl(L_\infty(X)\bigr).
\]
Consequently, $X/\RL(X)$ is the largest left-nilpotent quotient of
$X$. We describe the behavior of $\RL$ under epimorphisms, quotients
and direct products, its relation with subskew braces, and the
inclusions
\[
L_\infty(X)\leq \RL(X)\leq X^2.
\]
For Sylow theory and bounds on the left-nilpotency class of
left-nilpotent skew braces, we refer to \cite{EGGK}.

For skew braces of nilpotent type we obtain the second description
\[
\RL(X)=\Id_X\bigl(\gamma_\infty(A)\bigr).
\]
Put $Q=G/L_\infty(X)$, and let $\Delta$ be the group induced by
$\Lambda$ on $Q$. The image of $A$ in the affine group
$Q\rtimes\Delta$ acts faithfully and transitively on $Q$. We prove
that $L_\infty(X)$ is an ideal if and only if this action is regular.
Moreover, $Q\rtimes\Delta$ is nilpotent and hence is the direct product
of its Sylow subgroups. Consequently, a counterexample can exist only
if there is a prime $p$ such that
\[
p\mid\bigl|(L_\infty(X),\cdot)_{\mathrm{ab}}\bigr|
\qquad\text{and}\qquad
p^2\mid[X:L_\infty(X)].
\]
It follows that $p^3\mid|X|$; in particular, if $|X|$ is cube-free,
then $L_\infty(X)$ is an ideal of $X$.  We also prove a criterion of
a different nature, on the gap between $L_\infty(X)$ and its ideal
closure: if
\[
\gcd\bigl(|L_\infty(X)|,\,[\RL(X):L_\infty(X)]\bigr)=1,
\]
then $L_\infty(X)$ is an ideal of $X$.  Hence in any counterexample
the residual gap shares a prime divisor with $|L_\infty(X)|$.

The residual also gives two further structural directions. First, let
$I\trianglelefteq X$ and assume that $X/I$ is left nilpotent. Starting
with $I_0=I$ and defining $I_{n+1}=X*I_n$, we prove that this sequence
stabilizes at $L_\infty(X)$. Consequently, $X$ is left nilpotent if and
only if $I_n=0$ for some $n$. Second, iterating the residual produces a
descending chain of subskew braces. Its final term is the largest
perfect subskew brace of $X$, and the chain reaches $0$ exactly for
KSV-solvable skew braces.

The main part of the paper concerns supplements to $\RL(X)$. Let
$\pi$ be the set of prime divisors of $|\Lambda|$, and let
$N=O^{\pi'}(G)$ be the $\pi'$-residual of $G$. Thus $N$ is the
smallest normal subgroup of $G$ such that $G/N$ is a $\pi'$-group.
The subgroup $N$ is characteristic in $G$ and hence is a strong left
ideal of $X$. Let $\rho:G\to G/N$ be the natural map, and let
$\bar\Lambda$ be the group induced by $\Lambda$ on $G/N$. Since
$\bar\Lambda$ is a $\pi$-group and $G/N$ is a $\pi'$-group, this
action is coprime. Put
\[
C=\rho^{-1}\bigl(C_{G/N}(\bar\Lambda)\bigr).
\]
Using the standard coprime-action decomposition, we prove that
$C=\{c\in X:X*c\subseteq N\}$ is a left ideal of $X$ and that
\[
X=L_\infty(X)+C=\RL(X)+C.
\]
Thus $C$ is a supplement to $\RL(X)$. Moreover, $N$ is an ideal of
$C$ and $C/N$ is a trivial skew brace. The construction therefore
gives a supplement with additional internal structure.

By the Schur--Zassenhaus theorem for finite skew braces \cite{DameleSZ},
$\RL(X)$ has a left-nilpotent complement whenever $|\RL(X)|$ and
$[X:\RL(X)]$ are coprime. The Sylow and Hall theorems proved in
\cite{Truman} yield further supplements when $X/\RL(X)$ has
prime-power order and when both underlying groups of $X$ are solvable,
respectively (see \cite{CDDFT} for earlier results under additional
structural hypotheses). These conclusions do not hold without further
assumptions: a skew brace of order $18$ has no proper supplement to
its residual.

The paper is organized as follows.  Section~\ref{sec:prelim} contains
preliminaries.  Section~\ref{sec:dict} recalls the commutator
description of the left series.  Section~\ref{sec:final} collects the
properties of its final term.  Section~\ref{sec:res} develops the
properties of the left-nilpotent residual.  Section~\ref{sec:ideality}
studies the ideality problem for the final term, with particular
emphasis on skew braces of nilpotent type.  Section~\ref{sec:ext}
studies extensions and iteration, and Section~\ref{sec:supp} treats
supplements and complements of the residual.

\section{Preliminaries}\label{sec:prelim}
Throughout the paper, all skew braces are assumed to be finite unless
explicitly stated otherwise.
A \emph{skew (left) brace} is a triple $X=(X,+,\cdot)$ such that
$(X,+)$ and $(X,\cdot)$ are groups on the same set and
\[
a\cdot(b+c)=a\cdot b-a+a\cdot c
\qquad(a,b,c\in X).
\]
The common identity is denoted by $0$.  The lambda map
\[
\lambda:(X,\cdot)\longrightarrow\Aut(X,+),\qquad
\lambda_a(b)=-a+a\cdot b,
\]
is a group homomorphism, and the associated star product is
\[
a*b=\lambda_a(b)-b=-a+a\cdot b-b.
\]
For subsets $S,T\subseteq X$, the notation $S*T$ means the subgroup of
$(X,+)$ generated by all $s*t$ with $s\in S$ and $t\in T$.

A subgroup $I\le(X,+)$ is a \emph{left ideal} if it is invariant under
every $\lambda_a$.  A left ideal is automatically a subskew brace.  It
is \emph{strong} if $I\trianglelefteq(X,+)$, and it is an \emph{ideal}
if it is normal in both underlying groups.  We write $\Id_X(S)$ for
the smallest ideal of $X$ containing a subset $S$.

The left series is
\[
X^1=X,\qquad X^{n+1}=X*X^n\quad(n\geq1).
\]
The skew brace $X$ is \emph{left nilpotent} if $X^n=0$ for some $n$.
Each $X^n$ is a left ideal \cite[Proposition~12.2.2]{CedoVendraminBook}.
We shall also use the standard fact that
\begin{equation}\label{eq:X2ideal}
X^2=X*X\trianglelefteq X
\end{equation}
and that $X/X^2$ is the largest trivial quotient
\cite[Proposition~13.1.3]{CedoVendraminBook}.

A skew brace is of \emph{nilpotent type} if its additive group is
nilpotent.  The following theorem will be used repeatedly.

\begin{theorem}[Ced\'o--Smoktunowicz--Vendramin
{\cite[Theorem~4.8]{CSV}}]\label{thm:csv}
A finite skew brace of nilpotent type is left nilpotent if and only if
its multiplicative group is nilpotent.
\end{theorem}

Following Konovalov--Smoktunowicz--Vendramin \cite[\S5]{KSV} (see also
the erratum \cite{KSVErratum}), define
\[
D_0(X)=X,\qquad D_{n+1}(X)=D_n(X)*D_n(X).
\]
To distinguish this notion from other notions of solvability for skew
braces, we call $X$ \emph{KSV-solvable} if $D_n(X)=0$ for some $n$.
By \eqref{eq:X2ideal}, every $D_n(X)$ is a subskew brace. We call a
skew brace $C$ \emph{perfect} if $C*C=C$.

We recall the group-action result used below. Let a finite group $B$
act by automorphisms on a finite group $H$. Working in $H\rtimes B$,
put
$
[H,B]=\langle h^{-1}h^b:h\in H,\ b\in B\rangle
$
and define $[H,{}_1B]=[H,B]$ and
$[H,{}_{n+1}B]=[[H,{}_nB],B]$.

\begin{theorem}[Isaacs--Meierfrankenfeld]\label{thm:IM}
The subgroups $[H,{}_nB]$ are $B$-invariant and subnormal in $H$.
Their final term $K$ satisfies $[K,B]=K$. Moreover, $K$ is normal in
$H$ if and only if
$
\gamma_\infty([H,B])\leq K,
$
where $\gamma_\infty([H,B])$ denotes the nilpotent residual of
$[H,B]$.
\end{theorem}

\begin{proof}
See \cite[Theorems~D and E]{IM}.
\end{proof}

\section{The left series as a commutator series}\label{sec:dict}

Let
\[
G=(X,+),\qquad A=(X,\cdot),\qquad
\Lambda=\lambda(A)\leq\Aut(G).
\]
We regard $A$ as an operator group on $G$ via $\lambda$, so
commutators with $A$ are taken in $G\rtimes_\lambda A$. For $x\in G$
and $a\in A$, write
$
[x,a]=-x+\lambda_a(x).
$
Since $\ker\lambda$ centralizes $G$, commutators with $A$ and with
$\Lambda$ determine the same subgroups of $G$. In particular,
$[G,{}_nA]=[G,{}_n\Lambda]$ for every $n\geq1$.

\begin{lemma}\label{lem:dict}
For every subgroup $Y\leq G$,
$
X*Y=[Y,A]=[Y,\Lambda].
$
\end{lemma}

\begin{proof}
Let $a\in X$ and $y\in Y$. Then
\[
a*(-y)
=\lambda_a(-y)+y
=-\lambda_a(y)+y
=-\bigl(-y+\lambda_a(y)\bigr)
=-[y,a].
\]
As $y$ runs through $Y$, so does $-y$, and the two sets of generators
give the same subgroup.
\end{proof}

\begin{proposition}\label{prop:series}
For every $n\geq1$,
$
X^{n+1}
=[G,\underbrace{\Lambda,\ldots,\Lambda}_{n}]
=[G,{}_nA].
$
Moreover, every $X^n$ is a left ideal of $X$,
$X^2\trianglelefteq G$, and
$X^{n+1}\trianglelefteq(X^n,+)$ for every $n\geq1$.
Thus every term of the left series is subnormal in $G$.
\end{proposition}

\begin{proof}
By Lemma~\ref{lem:dict},
$
X^{n+1}=X*X^n=[X^n,A]=[X^n,\Lambda].
$
Starting with $X^1=G$ and applying this identity repeatedly gives the
displayed commutator formula. The remaining assertions are the
standard properties of repeated commutators recalled in
Section~\ref{sec:prelim}. Since the action of $A$ on $G$ is the
$\lambda$-action, $A$-invariance, $\Lambda$-invariance, and
$\lambda$-invariance are equivalent.
\end{proof}
\section{Properties of the final term of the left series}\label{sec:final}

\begin{definition}\label{def:final}
The final term of the left series of $X$ is
\[
L_\infty(X)=\bigcap_{n\ge1}X^n.
\]
\end{definition}

Since $X$ is finite, its left series eventually reaches $L_\infty(X)$.
Thus $X$ is left nilpotent if and only if $L_\infty(X)=0$.

\begin{proposition}\label{prop:idempotent}
The subgroup $L_\infty(X)$ is a left ideal of $X$ and
$
X*L_\infty(X)=L_\infty(X).
$
Moreover, $L_\infty(X)$ is the largest subgroup $Y\le(X,+)$ satisfying
$X*Y=Y$.  In particular, every subskew brace $C\le X$ satisfying
$C*C=C$ is contained in $L_\infty(X)$.
\end{proposition}

\begin{proof}
Since $X$ is finite, there is an $m$ such that
$X^m=X^{m+1}=L_\infty(X)$. By Proposition~\ref{prop:series}, $X^m$ is
a left ideal of $X$ and
$
X*L_\infty(X)=X*X^m=X^{m+1}=X^m=L_\infty(X).
$

Let $Y\leq(X,+)$ satisfy $X*Y=Y$. We prove by induction that
$Y\leq X^n$ for every $n\geq1$. This is clear for $n=1$. If
$Y\leq X^n$, then
$
Y=X*Y\leq X*X^n=X^{n+1}.
$
Therefore $Y\leq L_\infty(X)$.

Finally, let $C\leq X$ be a subskew brace satisfying $C*C=C$. Again,
$C\leq X^1$, and if $C\leq X^n$, then
$
C=C*C\leq X*X^n=X^{n+1}.
$
Hence $C\leq X^n$ for every $n$, and so $C\leq L_\infty(X)$.
\end{proof}

\begin{theorem}\label{thm:addnormal}
Let $\gamma_\infty(X^2,+)$ denote the nilpotent residual of the group
$(X^2,+)$.  Then $L_\infty(X)$ is normal in $(X,+)$ if and only if
$
\gamma_\infty(X^2,+)\le L_\infty(X).
$
\end{theorem}

\begin{proof}
By Proposition~\ref{prop:series}, $L_\infty(X)$ is the final
commutator of $G$ by $\Lambda$. Since $[G,\Lambda]=X^2$, the result
follows from Theorem~\ref{thm:IM}.
\end{proof}

\begin{corollary}\label{cor:addnormal}
If $(X^2,+)$ is nilpotent, in particular if $X$ is of nilpotent type,
then $L_\infty(X)$ is a strong left ideal.
\end{corollary}

\begin{proof}
Since $(X^2,+)$ is nilpotent, $\gamma_\infty(X^2,+)=0$. Hence
Theorem~\ref{thm:addnormal} shows that $L_\infty(X)$ is normal in
$(X,+)$. The result now follows from Proposition~\ref{prop:idempotent}.
\end{proof}

\begin{example}\label{ex:18}
Let $X=\SB{18}{20}$.  Then
$(X,+)\cong (X,\cdot)\cong S_3\times C_3$, and its left series has
sizes $18>9>3$,
with $L_\infty(X)=X^3$ and $(X^2,+)\cong C_3\times C_3$.  The subgroup
$L_\infty(X)$ is $\lambda$-invariant and normal in $(X,+)$, but it is
not normal in $(X,\cdot)$.  Hence $L_\infty(X)$ is not an ideal of
$X$.  Thus the final term of the left series of a finite skew brace
need not be an ideal.
\end{example}

\begin{example}\label{ex:36}
Let $X=\SB{36}{191}$.  The left series of $X$ has orders
$36>18>3$,
and $L_\infty(X)=X^3$ has order $3$.  Moreover, $(X^2,+)$ is
non-nilpotent of order $18$ and
$\bigl|\gamma_\infty(X^2,+)\bigr|=9$.  Consequently,
$\gamma_\infty(X^2,+)\nleq L_\infty(X)$,
and Theorem~\ref{thm:addnormal} shows that $L_\infty(X)$ is not normal
in $(X,+)$.  Thus the final term of the left series of a finite skew
brace need not be an additive normal subgroup.
\end{example}

\section{Properties of the left-nilpotent residual}\label{sec:res}

Example~\ref{ex:18} shows that $L_\infty(X)$ need not be an ideal and
therefore need not define a quotient skew brace. This leads to the
following definition.

\begin{definition}
The \emph{left-nilpotent residual} of $X$ is
\[
\RL(X)=\bigcap\{I:I\text{ is an ideal of }X
\text{ and }X/I\text{ is left nilpotent}\}.
\]
\end{definition}

\begin{lemma}\label{lem:quot}
Let $I$ be an ideal of $X$. Then $(X/I)^n=(X^n+I)/I$ for every
$n\geq1$. Consequently, $X/I$ is left nilpotent if and only if
$L_\infty(X)\leq I$.
\end{lemma}

\begin{proof}
Let $\rho:X\to X/I$ be the natural epimorphism. Since $\rho$ preserves
the star product, induction gives $\rho(X^n)=(X/I)^n$ for every
$n\geq1$. Since $\rho(X^n)=(X^n+I)/I$, the first assertion follows.
The quotient $X/I$ is left nilpotent if and only if $X^n\leq I$ for
some $n$. Since the left series of $X$ eventually reaches
$L_\infty(X)$, this is equivalent to $L_\infty(X)\leq I$.
\end{proof}

\begin{theorem}\label{thm:residual}
Let $X$ be a skew brace. Then the following statements hold.
\begin{enumerate}
\item
The left-nilpotent residual is the ideal generated by the final term
of the left series:
\[
\RL(X)=\Id_X\bigl(L_\infty(X)\bigr).
\]

\item
The quotient $X/\RL(X)$ is the largest left-nilpotent quotient of
$X$. More precisely, if $f:X\to Y$ is an epimorphism and $Y$ is left
nilpotent, then $\RL(X)\leq\ker f$.

\item We have
$
L_\infty(X)\leq\RL(X)\leq X^2.
$
\item The equality $\RL(X)=0$ holds if and only if $X$ is left
nilpotent, and $\RL(X)=L_\infty(X)$ holds if and only if
$L_\infty(X)$ is an ideal of $X$.
\end{enumerate}
\end{theorem}
\begin{proof}
By Lemma~\ref{lem:quot}, the ideals $I$ for which $X/I$ is left
nilpotent are precisely the ideals containing $L_\infty(X)$. Their
intersection is therefore the smallest ideal containing
$L_\infty(X)$, proving~\textup{(1)}.

It follows that $X/\RL(X)$ is left nilpotent. If $f:X\to Y$ is an
epimorphism and $Y$ is left nilpotent, then $X/\ker f\cong Y$. Hence
Lemma~\ref{lem:quot} gives $L_\infty(X)\leq\ker f$, and~\textup{(1)}
gives $\RL(X)\leq\ker f$. This proves~\textup{(2)}.

The first inclusion in~\textup{(3)} follows from~\textup{(1)}.
Moreover, $X^2$ is an ideal and $X/X^2$ is trivial, and so the definition
of the residual gives $\RL(X)\leq X^2$.

Finally, $\RL(X)=0$ if and only if $L_\infty(X)=0$, which is
equivalent to left nilpotency. If $L_\infty(X)$ is an ideal, then
\textup{(1)} gives $\RL(X)=L_\infty(X)$. The converse follows because
$\RL(X)$ is an ideal.
\end{proof}
\begin{remark}\label{rem:quotient-class}
For a nonzero left-nilpotent skew brace $B$, let $\ell(B)$ denote its
left-nilpotency class. Put $Q=X/\RL(X)$ and suppose that $Q\neq0$.
By \cite[Theorems~3.2 and~3.5]{EGGK}, the Sylow subgroups of
$(Q,\cdot)$ are the multiplicative groups of the Sylow subskew braces
of $Q$, and
\[
\ell(Q)\leq
\max\{\ell(P):P\text{ is a Sylow subskew brace of }Q\}+1.
\]
\end{remark}

\begin{proposition}\label{prop:epi}
Let $f:X\to Y$ be an epimorphism. Then
$
f\bigl(L_\infty(X)\bigr)=L_\infty(Y)$ and 
$f(\RL(X))=\RL(Y).
$
Consequently, if $I$ is an ideal of $X$, then
$\RL(X/I)=(\RL(X)+I)/I$. Both $L_\infty(X)$ and $\RL(X)$ are
invariant under every skew brace automorphism of $X$.
\end{proposition}

\begin{proof}
Since $f$ is surjective and preserves the star product, induction
gives $f(X^n)=Y^n$ for every $n$. Passing to the final terms gives
$f(L_\infty(X))=L_\infty(Y)$. By
Theorem~\ref{thm:residual}\textup{(1)},
\[
f\bigl(\RL(X)\bigr)
=f\bigl(\Id_X(L_\infty(X))\bigr)
=\Id_Y(L_\infty(Y))
=\RL(Y).
\]
The remaining assertions follow by applying these equalities to the
natural quotient map and to automorphisms of $X$.
\end{proof}

\begin{proposition}\label{prop:prod}
For skew braces $X$ and $Y$, we have
$
L_\infty(X\times Y)=L_\infty(X)\times L_\infty(Y)$ and\\ $\RL(X\times Y)=\RL(X)\times\RL(Y).
$
\end{proposition}

\begin{proof}
Since the star product is defined componentwise,
$
(X\times Y)^n=X^n\times Y^n
$
for every $n\geq1$. Passing to the final terms gives the first
equality. Let now $R=\RL(X\times Y)$. The ideal\\
$\RL(X)\times\RL(Y)$ contains
$L_\infty(X)\times L_\infty(Y)$, and so
Theorem~\ref{thm:residual}\textup{(1)} gives
$R\leq\RL(X)\times\RL(Y)$.

Conversely, $R\cap(X\times\{0\})$ is an ideal of $X\times\{0\}$
containing $L_\infty(X)\times\{0\}$. It follows that
$\RL(X)\times\{0\}\leq R$. Similarly,
$\{0\}\times\RL(Y)\leq R$. Since $R$ is an additive subgroup,
$\RL(X)\times\RL(Y)\leq R$, proving the second equality.
\end{proof}
\begin{lemma}\label{lem:subskew-residual}
Let $Y\leq X$ be a subskew brace. Then $Y^n\leq X^n$ for every
$n\geq1$. In particular, every subskew brace of a left-nilpotent skew
brace is left nilpotent, and
$
\RL(Y)\leq Y\cap\RL(X).
$
\end{lemma}

\begin{proof}
The first assertion follows by induction, since
$
Y^{n+1}=Y*Y^n\leq X*X^n=X^{n+1}.
$
Hence the image of $Y$ in $X/\RL(X)$ is left nilpotent. Since this
image is isomorphic to $Y/(Y\cap\RL(X))$, the definition of the
residual gives $\RL(Y)\leq Y\cap\RL(X)$.
\end{proof}
\section{The ideality problem for skew braces of nilpotent type}
\label{sec:ideality}

Throughout this section, $X$ denotes a finite skew brace of nilpotent
type. We put
\[
G=(X,+),\qquad A=(X,\cdot),\qquad
\Lambda=\lambda(A),\qquad L=L_\infty(X).
\]
Since $X$ is of nilpotent type, Corollary~\ref{cor:addnormal} shows
that $L$ is a strong left ideal. We investigate whether $L$ must
always be an ideal of $X$.

\begin{conjecture}\label{conj:ideal}
For every finite skew brace $X$ of nilpotent type,
$L_\infty(X)$ is an ideal of $X$. Equivalently,
$
\RL(X)=L_\infty(X).
$
\end{conjecture}

We begin by comparing the left-nilpotent residual with the nilpotent
residual of the multiplicative group.

\begin{theorem}\label{thm:twodesc}
We have
$
\RL(X)=\Id_X(L)=\Id_X\bigl(\gamma_\infty(A)\bigr).
$
Equivalently, $\RL(X)$ is the smallest ideal $I$ of $X$ such that
$A/I$ is nilpotent.
\end{theorem}

\begin{proof}
The equality $\RL(X)=\Id_X(L)$ follows from
Theorem~\ref{thm:residual}. Let $I$ be an ideal of $X$. Since
$(X/I,+)$ is nilpotent, Theorem~\ref{thm:csv} shows that $X/I$ is
left nilpotent if and only if $A/I$ is nilpotent. The latter is
equivalent to $\gamma_\infty(A)\leq I$. Taking the intersection of
all such ideals gives
$\RL(X)=\Id_X(\gamma_\infty(A))$.
\end{proof}

\begin{corollary}\label{cor:nilptype}
The following statements hold.
\begin{enumerate}
\item If $L$ is an ideal of $X$, then $\RL(X)=L$ and
$\gamma_\infty(A)\leq L$.
\item If $\gamma_\infty(A)$ is an ideal of $X$, then
$\RL(X)=\gamma_\infty(A)$ and $L\leq\gamma_\infty(A)$.
\item If both $L$ and $\gamma_\infty(A)$ are ideals of $X$, then
$L=\RL(X)=\gamma_\infty(A)$.
\end{enumerate}
\end{corollary}

\begin{proof}
Apply Theorem~\ref{thm:twodesc} and the minimality of the ideal
generated by a subset.
\end{proof}

Let $Q=G/L$ and let $\Delta\leq\Aut(Q)$ be the group induced by
$\Lambda$ on $Q$. For every prime $p$, let $G_p$ be the Sylow
$p$-subgroup of $G$, and set
$
L_p=L\cap G_p,
\quad
Q_p=G_p/L_p.
$
Since $G$ is nilpotent,
$
G=\prod_pG_p
\quad\text{and}\quad
L=\prod_pL_p.
$
Hence there is a natural isomorphism
\[
Q=G/L\cong\prod_pG_p/L_p=\prod_pQ_p.
\]
From now on we identify $Q$ with $\prod_pQ_p$ through this
isomorphism. Since the Sylow subgroups $G_p$ are characteristic in
$G$, they are $\Lambda$-invariant.

\begin{proposition}\label{prop:nilptype-reduction}
The following conditions are equivalent:
\begin{enumerate}
\item $L$ is an ideal of $X$;
\item $L*X\leq L$;
\item $\lambda_a(x)-x\in L$ for all $a\in L$ and $x\in X$;
\item $\lambda_a$  acts trivially on $G/L$ for all $a\in L$;
\item $[G_p,\lambda_a]\leq L_p$ for every prime $p$ and every
$a\in L$.
\end{enumerate}
\end{proposition}

\begin{proof}
Since $L$ is a strong left ideal,
\cite[Proposition~2.2]{DE} gives the equivalence of (1) and
(2). The equivalence of (2), (3) and (4) follows from
$a*x=\lambda_a(x)-x$. Finally, under the identification above, the automorphism of $Q$
induced by $\lambda_a$ is trivial if and only if its induced
automorphism on every $Q_p$ is trivial. For a fixed prime $p$, the
latter condition is equivalent to
$
[G_p,\lambda_a]\leq L_p.
$
This proves the equivalence of \textup{(4)} and \textup{(5)}.
\end{proof}

\begin{lemma}\label{lem:induced-p-action}
For every prime $p$, the group
$\Delta_p=\Delta/C_\Delta(Q_p)$ induced by $\Delta$ on $Q_p$ is a
$p$-group.
\end{lemma}

\begin{proof}
Since the left series of $X$ eventually reaches $L$, there exists
$m$ such that $[Q,{}_m\Delta]=1$. Hence
$[Q_p,{}_m\Delta_p]=1$. Let $\alpha\in\Delta_p$ have order coprime to $p$. By
\cite[Lemma~4.29]{Isaac},
$
[Q_p,\alpha,\alpha]=[Q_p,\alpha].
$ It follows that
$[Q_p,{}_n\alpha]=[Q_p,\alpha]$ for every $n\geq1$. Since the
left-hand side is trivial for sufficiently large $n$, we obtain
$[Q_p,\alpha]=1$. As the action of $\Delta_p$ on $Q_p$ is faithful, we have
$\alpha=1$. Thus $\Delta_p$ is a $p$-group.
\end{proof}

\begin{proposition}
\label{prop:delta-product}
With the notation above, we have
$
\Delta\cong\prod_p\Delta_p$ and $Q\rtimes\Delta\cong\prod_p(Q_p\rtimes\Delta_p).
$
Under these identifications, the Sylow $q$-subgroup of $\Delta$
centralizes $Q_p$ whenever $q\neq p$.
\end{proposition}

\begin{proof}
The natural map $\Delta\to\prod_p\Delta_p$ is injective and is
surjective on every factor. Since the orders of the factors are
pairwise coprime, $|\Delta|$ is divisible by $\prod_p|\Delta_p|$.
The reverse divisibility follows from injectivity, and so the map is an
isomorphism. Under the isomorphism $\Delta\cong\prod_p\Delta_p$, an element of the
factor $\Delta_p$ has trivial components in all the other factors.
Therefore $\Delta_p$ acts trivially on $Q_q$ whenever $q\neq p$.
Thus the action of $\Delta$ on $Q=\prod_pQ_p$ is componentwise, and
hence
$
Q\rtimes\Delta\cong\prod_p(Q_p\rtimes\Delta_p).
$
The final claim follows immediately.
\end{proof}

The next result strengthens the inclusion in
Corollary~\ref{cor:nilptype}(1). It does not require $L$ to be an
ideal.

\begin{theorem}\label{thm:core-nilpotent}
Let $C=\operatorname{Core}_A(L)$. Then $A/C$ is nilpotent and
\[
C=\{a\in L:\lambda_a\text{ acts trivially on }G/L\}.
\]
Consequently,
$
\gamma_\infty(A)\leq C\leq L.
$
\end{theorem}

\begin{proof}
By Proposition~\ref{prop:delta-product},
$Q\rtimes\Delta$ is the direct product of the $p$-groups
$Q_p\rtimes\Delta_p$, and hence is nilpotent. Define
\[
\theta\colon A\longrightarrow Q\rtimes\Delta,\qquad
\theta(a)=(a+L,\overline{\lambda_a}),
\] where 
 $\overline{\lambda_a}$ denotes the automorphism induced by
$\lambda_a$ on $Q$. The multiplication in $Q\rtimes\Delta$ is given by
\[
(q,\alpha)(r,\beta)
=
\bigl(q+\alpha(r),\alpha\beta\bigr).
\]
Hence, for $a,b\in A$,
\[
\begin{aligned}
\theta(a)\theta(b)
&=(a+L,\overline{\lambda_a})
  (b+L,\overline{\lambda_b})\\
&=\bigl(a+L+\overline{\lambda_a}(b+L),
        \overline{\lambda_a}\,\overline{\lambda_b}\bigr)\\
&=\bigl(a+\lambda_a(b)+L,
        \overline{\lambda_{a\cdot b}}\bigr)\\
&=\bigl(a\cdot b+L,
        \overline{\lambda_{a\cdot b}}\bigr)\\
&=\theta(a\cdot b).
\end{aligned}
\]
Here we used $a\cdot b=a+\lambda_a(b)$ and the fact that
$\lambda\colon A\to\operatorname{Aut}(G)$ is a group homomorphism.
Therefore $\theta$ is a group homomorphism. Its kernel is
$
K=\{a\in L:\lambda_a\text{ acts trivially on }G/L\}.
$
Since $K\trianglelefteq A$ and $K\leq L$, we have $K\leq C$.

Conversely, let $c\in C$ and $x\in X$. Since $C\trianglelefteq A$,
the element $c'=x^{-1}\cdot c\cdot x$ belongs to $C\leq L$, and
$c\cdot x=x\cdot c'$. Hence
$c+\lambda_c(x)=x+\lambda_x(c')$ which yields that $c*x=\lambda_c(x)-x\in L$ for every $x\in X$. Then
$
\lambda_c(x)+L=x+L
$
for every $x\in X$, and so $\lambda_c$ acts trivially on
$G/L$, that is $c\in K$. It follows that $C=K$.

Finally, $A/C\cong\theta(A)$ is a subgroup of the nilpotent group
$Q\rtimes\Delta$. Therefore $A/C$ is nilpotent and
$\gamma_\infty(A)\leq C\leq L$.
\end{proof}

\begin{corollary}
\label{cor:regularity}
Let $H=\theta(A)$. The semidirect product $Q\rtimes\Delta$ acts on the underlying set of
$Q$ by
$
(q,\delta)\cdot r=q+\delta(r)
\quad(q,r\in Q,\ \delta\in\Delta).
$ Under this action, the group $H$ acts faithfully and
transitively, and
$
\operatorname{Stab}_H(0+L)=\theta(L)\cong L/C.
$
Consequently, $L$ is an ideal of $X$ if and only if $H$ acts
regularly on $Q$.
\end{corollary}

\begin{proof}
Clearly this action is faithful. Since
$
\theta(a)=(a+L,\overline{\lambda_a}),
$
we have
\[
\theta(a)\cdot(0+L)
=(a+L)+\overline{\lambda_a}(0+L)
=a+L.
\]
Since $A$ and $G$ have the same underlying set $X$, every element of
$Q=G/L$ has the form $a+L$ for some $a\in A$. As
$
\theta(a)\cdot(0+L)=a+L,
$
the orbit of $0+L$ under $H=\theta(A)$ is all of $Q$. Hence $H$ acts
transitively on $Q$. Moreover,
$\theta(a)$ fixes $0+L$ if and only if $a\in L$. Hence
$\operatorname{Stab}_H(0+L)=\theta(L)$. By
Theorem~\ref{thm:core-nilpotent}, the kernel of
$\theta|_L$ is $C$ and hence
$
\theta(L)\cong L/C.
$
Thus the action of $H$ is regular if and only if $L=C$, equivalently
if and only if $L\trianglelefteq A$. Since $L$ is a strong left ideal,
this is equivalent to $L$ being an ideal of $X$.
\end{proof}

\begin{proposition}\label{prop:local-affine}
Under the identification above, let
$H_p$ be the Sylow $p$-subgroup of $H=\theta(A)$. Then
$
H=\prod_pH_p$, $H_p\leq Q_p\rtimes\Delta_p,
$
and $H_p$ acts transitively on $Q_p$. Consequently, $L$ is an ideal
of $X$ if and only if every $H_p$ acts regularly on $Q_p$.
\end{proposition}
\begin{proof}
The group $H$ is a subgroup of the finite nilpotent group
$\prod_p(Q_p\rtimes\Delta_p)$, and so it is nilpotent and is the direct
product of its Sylow subgroups. A $p$-subgroup has trivial projection
to every $q$-factor with $q\neq p$, and hence
$H_p\leq Q_p\rtimes\Delta_p$.

Since $H$ is transitive on $Q=\prod_pQ_p$ and only $H_p$ can move
the $p$-component, $H_p$ acts transitively on $Q_p$. Moreover, the
action of $H$ on $Q$ is regular if and only if the action of every
$H_p$ on $Q_p$ is regular. The final assertion now follows from Corollary~6.8.\end{proof}

\begin{corollary}
\label{cor:multres-ideal}
If $\gamma_\infty(A)$ is an ideal of $X$, then
$
L=\gamma_\infty(A)=\RL(X).
$
In particular, $L$ is an ideal of $X$.
\end{corollary}

\begin{proof}
By Theorem~\ref{thm:twodesc},
$
\RL(X)=\Id_X\bigl(\gamma_\infty(A)\bigr)
      =\gamma_\infty(A).
$
On the other hand, Theorem~\ref{thm:core-nilpotent} and
Theorem~\ref{thm:residual}(3) give
$
\gamma_\infty(A)\leq L\leq\RL(X).
$
Hence all three groups are equal.
\end{proof}

\begin{theorem}\label{thm:abelianized-coprime}
If $
\gcd\bigl(|(L,\cdot)_{\mathrm{ab}}|,[X:L]\bigr)=1,
$
then $L$ is an ideal of $X$, and
$
\RL(X)=L.
$
\end{theorem}

\begin{proof}
For each prime $p$ and each $a\in L$, the automorphism $\lambda_a$
induces an automorphism of $Q_p=G_p/L_p$ given by
$
g+L_p\longmapsto\lambda_a(g)+L_p.
$
Since $\lambda_{a\cdot b}=\lambda_a\lambda_b$, this defines a group
homomorphism
$
\varphi_p\colon (L,\cdot)\longrightarrow\Delta_p.
$

Suppose that $p\mid[X:L]$. By Lemma~\ref{lem:induced-p-action},
$\Delta_p$ is a $p$-group. If $\varphi_p\ne 1$, its image
would be a nontrivial $p$-group and hence would have a quotient of
order $p$. It would follow that $(L,\cdot)$ has a quotient of order
$p$, and therefore
$
p\mid |(L,\cdot)_{\mathrm{ab}}|,
$ which is a contradiction. Hence $\varphi_p=1$.

If $p\nmid[X:L]$, then $Q_p=1$ and so the action on $Q_p$ is trivial.
Therefore every element of $L$ acts trivially on
$
Q=\prod_p Q_p.
$
Proposition~\ref{prop:nilptype-reduction} now shows that $L$ is an
ideal of $X$. Finally, Theorem~\ref{thm:residual}(4) gives
$\RL(X)=L$.
\end{proof}

\begin{corollary}
\label{cor:squarefree-final-index}
If $[X:L]$ is square-free, then
$
X^2=L=\RL(X).
$
In particular, $L$ is an ideal of $X$.
\end{corollary}

\begin{proof}
Since $|Q|=[X:L]$ is square-free, every nontrivial $Q_p$ has order
$p$. By Lemma~\ref{lem:induced-p-action}, the group induced by
$\Delta$ on $Q_p$ is a $p$-group. Since
$|\operatorname{Aut}(Q_p)|=p-1$, this induced action is trivial.
Hence $\Delta=1$. It means that every $\lambda_a$ acts trivially on $G/L$.
Therefore, for all $a,x\in X$, $
\lambda_a(x)+L=x+L,
$
and hence
$
a*x=\lambda_a(x)-x\in L.
$
Thus we have $X*X\leq L$. Since $L\leq X^2=X*X$, we obtain
$X^2=L$. Hence $L*X\leq L$, and
\cite[Proposition~2.2]{DE} shows that $L$ is an ideal. Finally,
$\RL(X)=L$ by Theorem~\ref{thm:residual}(4).
\end{proof}

\begin{corollary}[Necessary conditions for a counterexample]
\label{cor:local-obstruction}
If $L$ is not an ideal of $X$, then there exists a prime $p$ such that
$
p\mid |(L,\cdot)_{\mathrm{ab}}|$ and $
p^2\mid[X:L].
$
Consequently, $p^3\mid|X|$. In particular, if $|X|$ is cube-free,
then $L$ is an ideal of $X$ and
$
\RL(X)=L.
$
\end{corollary}

\begin{proof}

Since $L$ is not an ideal, Proposition~\ref{prop:nilptype-reduction}
shows that the image of $(L,\cdot)$ in
$\Delta_p$ is nontrivial for some prime $p$. By Lemma~\ref{lem:induced-p-action}, this
image is a nontrivial $p$-group. Hence
$
p\mid |(L,\cdot)_{\mathrm{ab}}|.
$

The same image acts nontrivially on $Q_p$. If $|Q_p|=p$, then $|\Aut(Q_p)|=p-1$ which is impossible. Therefore
$
p^2\mid |Q_p|\mid [X:L].
$
Since $p\mid |L|$, it follows that $p^3\mid |X|$. The final claim
now follows from Theorem~\ref{thm:residual}\textup{(4)}.
\end{proof}
The preceding results give criteria in terms of the action of
$(L,\cdot)$ on $G/L$. We now give an independent criterion involving
the index of $L$ in its ideal closure. It is convenient to prove it
for an arbitrary strong left ideal containing $\gamma_\infty(A)$.

\begin{theorem}\label{thm:coprime-gap}
Let $X$ be a skew brace of nilpotent type and let $K$ be a
strong left ideal of $X$ containing $\gamma_\infty(A)$.  If
$
\gcd\bigl(|K|,\,[\,\Id_X(K):K\,]\bigr)=1,
$
then $K$ is an ideal of $X$; in particular $\Id_X(K)=K$.
\end{theorem}

\begin{proof}
Set 
$R=\Id_X(K)$
and $N=\gamma_\infty(A).$ Since $K$ is a strong left ideal, the construction used in the proof
of Theorem~\ref{thm:core-nilpotent} gives a homomorphism
\[
\theta_K\colon A\longrightarrow (G/K)\rtimes\Delta_K,
\qquad
\theta_K(a)=(a+K,\overline{\lambda_a}),
\]
where $\Delta_K$ is the group induced by $\Lambda$ on $G/K$.
The same calculation, using $N\trianglelefteq A$ and $N\leq K$,
shows that $N\leq\ker\theta_K$. Hence
$
H=\theta_K(A)
$
is a homomorphic image of the nilpotent group $A/N$ and is therefore
nilpotent.

Set
$
W=R/K$, $ H_0=\theta_K(R)$, and $P=\theta_K(K)$. For every $r\in R$,
$
\theta_K(r)\cdot(0+K)=r+K.
$
As $r$ runs through $R$, the cosets $r+K$ run through all the
elements of $W=R/K$. Hence the orbit of $0+K$ under
$H_0=\theta_K(R)$ is $W$. Moreover, $\theta_K(r)$ fixes $0+K$ if and
only if
$
r+K=0+K,
$
that is $r\in K$. Therefore
$
\operatorname{Stab}_{H_0}(0+K)=\theta_K(K)=P.
$ Consequently,
$
[H_0:P]=|W|=[R:K].
$

Let $\pi$ be the set of prime divisors of $|K|$. Since $P$ is a
homomorphic image of $(K,\cdot)$, it is a $\pi$-group. The hypothesis
implies that $[H_0:P]$ is a $\pi'$-number. Thus $P$ is a Hall
$\pi$-subgroup of the nilpotent group $H_0$ and hence contains every
$\pi$-element of $H_0$.

Let $a\in X$ and $k\in K$, and put
$
r=a\cdot k\cdot a^{-1}.
$
Since $(R,\cdot)\trianglelefteq A$, we have $r\in R$. Moreover,\\ $r$
has the same multiplicative order as $k$, and so $\theta_K(r)$ is a
$\pi$-element of $H_0$. Therefore $\theta_K(r)\in P$, and hence
$
r+K=\theta_K(r)\cdot(0+K)=0+K.
$
Thus $r\in K$, proving that $(K,\cdot)\trianglelefteq A$. Since $K$
is a strong left ideal, it is an ideal of $X$. Therefore we have
$\Id_X(K)=K$.
\end{proof}

\begin{corollary}\label{cor:gap-L}
If
$
\gcd\bigl(|L|,\,[\RL(X):L]\bigr)=1,
$
then $L$ is an ideal of $X$ and
$
\RL(X)=L.
$
\end{corollary}

\begin{proof}
By Corollary~\ref{cor:addnormal}, $L$ is a strong left ideal, and
$\gamma_\infty(A)\leq L$ by Theorem~\ref{thm:core-nilpotent}.  Since
$\RL(X)=\Id_X(L)$ by Theorem~\ref{thm:residual}(1), the claim
follows from Theorem~\ref{thm:coprime-gap} applied to $K=L$.
\end{proof}

\section{Extensions and iteration of the residual}\label{sec:ext}

\begin{definition}\label{def:relative}
Let $I$ be an ideal of $X$.  Define
$
I_0^X=I$ and $I_{n+1}^X=X*I_n^X\quad\text {for all}\quad n\ge0.
$
We call $(I_n^X)$ the \emph{left series of $I$ relative to $X$}. Note that each $I_n^X$ is a left ideal of $X$ and
$I_{n+1}^X\le I_n^X$.
\end{definition} 

\begin{theorem}\label{thm:relative}
Let $I$ be an ideal of $X$ and assume that $X/I$ is left nilpotent.  Then
the left series of $I$ relative to $X$ eventually reaches
$L_\infty(X)$.
Consequently, $X$ is left nilpotent if and only if $X/I$ is left
nilpotent and $I_n^X=0$ for some $n$.
\end{theorem}

\begin{proof}
By Lemma~\ref{lem:quot}, $L_\infty(X)\le I=I_0^X$.  If
$L_\infty(X)\le I_n^X$, then
\[
L_\infty(X)=X*L_\infty(X)\le X*I_n^X=I_{n+1}^X.
\]
Thus $L_\infty(X)\le I_n^X$ for every $n$.  On the other hand,
induction gives $I_n^X\le X^{n+1}$.  For all sufficiently large $n$,
$X^{n+1}=L_\infty(X)$, and hence $I_n^X=L_\infty(X)$.
\end{proof}

\begin{remark}\label{rem:notextension}
Finite left-nilpotent skew braces are not closed under extensions.
Let $X=\SB{18}{20}$ and put $I=X^2$. Now $X/I$ is the trivial skew
brace of order $2$, and hence is left nilpotent. Since $|I|=9$, both
underlying groups of $I$ are $3$-groups, so $I$ is left nilpotent by
Theorem~\ref{thm:csv}. But $X$ is not left nilpotent because
$L_\infty(X)\ne0$.
\end{remark}

\begin{definition}\label{def:iterated}
Set $R_0(X)=X$ and $R_{n+1}(X)=\RL(R_n(X))$ for all $n\geq0$.
Then $R_{n+1}(X)$ is an ideal of $R_n(X)$ and hence $(R_n(X))$ is a
descending chain of subskew braces such that every factor
$
R_n(X)/R_{n+1}(X)
$
is left nilpotent.
\end{definition}

\begin{proposition}\label{prop:perfectcore}
For every finite skew brace $X$, the chain $(R_n(X))$ eventually
reaches a subskew brace $P_L(X)$ satisfying
$
P_L(X)*P_L(X)=P_L(X).
$
Moreover $P_L(X)$ is the largest subskew brace $C\le X$ satisfying
$C*C=C$.
\end{proposition}

\begin{proof}
For every finite skew brace $Y$,
$
\RL(Y)\le Y^2=Y*Y
$
by Theorem~\ref{thm:residual}.  Hence the finite chain eventually
becomes constant.  If $\RL(P)=P$, then
$P\le P*P\le P$ whence $P*P=P$.

Conversely, let $C\le X$ satisfy $C*C=C$.  Now Proposition~\ref{prop:idempotent}
gives
\[
C\le L_\infty(X)\le\RL(X)=R_1(X).
\]
Applying the same argument inside $R_1(X)$, by induction we get
$C\le R_n(X)$ for every $n$.  Hence $C\le P_L(X)$.
\end{proof}

\begin{theorem}\label{thm:res-solv}
For each $n\ge0$,
$
R_n(X)\le D_n(X),
$
and the final terms of $(R_n(X))$ and $(D_n(X))$ coincide with
$P_L(X)$.  In particular, the following are equivalent.
\begin{enumerate}
\item $X$ is KSV-solvable;
\item $R_n(X)=0$ for some $n$;
\item $P_L(X)=0$.
\end{enumerate}

\end{theorem}

\begin{proof}
The inclusion $R_n(X)\le D_n(X)$ follows by induction.  If it holds
for $n$, then
\[
R_{n+1}(X)=\RL(R_n(X))
\le R_n(X)*R_n(X)
\le D_n(X)*D_n(X)=D_{n+1}(X).
\]
Since $X$ is finite, the star-derived series eventually reaches a
subskew brace $D$ with $D*D=D$.  Every subskew brace $C\le X$ with
$C*C=C$ is contained in every $D_n(X)$.  Thus the final term of the
star-derived series is the largest perfect subskew brace, namely
$P_L(X)$.  The
remaining assertions follow.
\end{proof}

\begin{definition}\label{def:reslength}
If $X$ is KSV-solvable, define the \emph{left-nilpotent residual
length}
\[
\ell_L(X)=\min\{n:R_n(X)=0\}.
\]
\end{definition}

\begin{corollary}\label{cor:length}
If $d_*(X)$ is the least $n$ such that $D_n(X)=0$, then
$
\ell_L(X)\le d_*(X).
$
\end{corollary}

\begin{remark}\label{rem:iterated-invariant}
Every $R_n(X)$ is invariant under every skew brace automorphism of
$X$. Indeed, $R_1(X)=\RL(X)$ is characteristic. If $R_n(X)$ is
invariant under $\alpha\in\Aut_{\mathrm{br}}(X)$, then
$\alpha|_{R_n(X)}$ is a skew brace automorphism of $R_n(X)$, and the
residual of $R_n(X)$ is characteristic in $R_n(X)$. Hence
$
\alpha(R_{n+1}(X))=R_{n+1}(X).
$
This does not by itself show that $R_n(X)$ is an ideal of $X$. We do not know whether $R_n(X)$ is always an ideal of $X$.

\end{remark}

\section{Supplements and complements of the residual}
\label{sec:supp}

Set $R=\RL(X)$ and $Q=X/R$. Then $Q$ is left nilpotent. A subskew
brace $C\leq X$ is a \emph{supplement} to $R$ if the natural map
$C\to X/R$ is surjective, equivalently if
$
X=R+C.
$
Since $R$ is an ideal, the natural quotient map
$
\rho\colon X\longrightarrow X/R
$
is a homomorphism for both group structures, and
\[
\rho^{-1}(\rho(C))=R+C=R\cdot C.
\]
Here $R+C$ and $R\cdot C$ are set products in the additive and
multiplicative groups, respectively. Hence the supplement condition
is also equivalent to $X=R\cdot C$. If, in addition, $R\cap C=0$,
then $C$ is a \emph{complement} to $R$.

\subsection{A supplement from the lambda action}

\begin{definition}\label{def:C-supplement}
Let $N$ be a strong left ideal of a skew brace $X$. Put $G=(X,+)$ and
$\Lambda=\lambda(A)$.
Let $\rho\colon G\to G/N$ be the natural map, and let $\bar\Lambda$
be the group induced by $\Lambda$ on $G/N$. Set $C
 =\rho^{-1}\bigl(C_{G/N}(\bar\Lambda)).$ Notice that
\[
C=\{c\in X:\lambda_a(c)+N=c+N
        \text{ for all }a\in X\}=\{c\in X:X*c\subseteq N\}.
\]
Then $N\leq C$ and
$
C/N=C_{G/N}(\bar\Lambda).
$
\end{definition}

For a finite group $H$, write $\pi(H)$ for the set of prime divisors
of $|H|$. For a set of primes $\sigma$, write $O^\sigma(H)$ for the
$\sigma$-residual of $H$, that is, the smallest normal subgroup
$K\trianglelefteq H$ such that $H/K$ is a $\sigma$-group.

\begin{theorem}\label{thm:coprime-action-supp}
Let $G=(X,+)$, $\Lambda=\lambda(X)\leq\operatorname{Aut}(G)$,
$\pi=\pi(\Lambda)$ and
$
N=O^{\pi'}(G).
$
Then $N$ is a strong left ideal of $X$. For $C$ as in
Definition~\ref{def:C-supplement},
$
C=\{c\in X:X*c\subseteq N\}
$
is a left ideal of $X$, and
\[
X=L_\infty(X)+C
 =L_\infty(X)\cdot C.
\]
Consequently,
\[
X=\RL(X)+C
 =\RL(X)\cdot C.
\]
Moreover, $N$ is an ideal of $C$ and $C/N$ is a trivial skew brace.
\end{theorem}

\begin{proof}
The subgroup $N=O^{\pi'}(G)$ is characteristic in $G$ and
$\Lambda$-invariant. Hence $N$ is a strong left ideal of $X$.

Let $\rho\colon G\to G/N$ be the natural map, and let $\bar\Lambda$
be the group induced by $\Lambda$ on $G/N$. Since $\bar\Lambda$ is a
$\pi$-group and $G/N$ is a $\pi'$-group, the action of
$\bar\Lambda$ on $G/N$ is coprime. The standard coprime-action identities
\cite[Lemma~4.29 and Theorem~4.34]{Isaac} give
$
G/N=[G/N,\bar\Lambda]C_{G/N}(\bar\Lambda)
$
and
$
[G/N,\bar\Lambda,\bar\Lambda]
=[G/N,\bar\Lambda].
$
Thus $[G/N,\bar\Lambda]$ is the final term of the repeated
commutator series of $G/N$ under the action of $\bar\Lambda$. By the definition of $C$,
 $N\leq C$ and
$
C/N=C_{G/N}(\bar\Lambda).
$
Since the fixed-point subgroup $C_{G/N}(\bar\Lambda)$ is
$\bar\Lambda$-invariant, its inverse image $C$ is a
$\Lambda$-invariant subgroup of $G$ whence $C$ is a left ideal of
$X$.

By Proposition~\ref{prop:series}, the image of $L_\infty(X)$ in
$G/N$ is the final commutator of $G/N$ under $\bar\Lambda$.
Therefore
$
\rho\bigl(L_\infty(X)\bigr)=[G/N,\bar\Lambda].
$
The coprime-action decomposition gives
\[
G/N
 =\rho\bigl(L_\infty(X)\bigr)(C/N).
\]
Since $N\leq C$, taking inverse images yields
$X=L_\infty(X)+C.
$ As $C$ is a left ideal, $\lambda_x(C)=C$ for every $x\in X$.
Consequently,
$
x\cdot C=x+\lambda_x(C)=x+C,
$
and hence
$
X=L_\infty(X)\cdot C.
$
Since $L_\infty(X)\leq\RL(X)$, it follows that
$
X=\RL(X)+C=\RL(X)\cdot C.
$
Finally, the definition of $C$ gives
$
X*C\subseteq N.
$
In particular,
$
N*C\subseteq N$ and $
C*C\subseteq N.
$
Since $N$ is a strong left ideal of $C$ and $N*C\subseteq N$,
\cite[Proposition~2.2]{DE} shows that $N$ is an ideal of $C$.
Moreover, $C*C\subseteq N$, and hence $C/N$ is a trivial skew
brace.
\end{proof}

\begin{corollary}\label{cor:Cpi-residual}
With the notation of Theorem~\ref{thm:coprime-action-supp},
$
\RL(C)\leq N\cap\RL(X).
$
In particular, if $N\cap\RL(X)=0$, then $C$ is left nilpotent
and hence is a left-nilpotent supplement to $\RL(X)$.
\end{corollary}

\begin{proof}
Since $N$ is an ideal of $C$ and $C/N$ is trivial, the quotient
$C/N$ is left nilpotent. Therefore
$
\RL(C)\leq N.
$
Since $C\leq X$, Lemma~\ref{lem:subskew-residual} gives
$
\RL(C)\leq C\cap\RL(X)
$ and the result follows.
\end{proof}

\subsection{Existence results}

We next give several existence results for supplements and
complements of the residual. Set $R=\RL(X)$ and $Q=X/R.$

\begin{proposition}\label{prop:coprime}
If
$
\gcd\bigl(|R|,[X:R]\bigr)=1,
$
then $R$ has a left-nilpotent complement in $X$.
\end{proposition}

\begin{proof}
By the Schur--Zassenhaus theorem for finite skew braces
\cite{DameleSZ}, the ideal $R$ has a subskew brace complement $C$.
Then
$
C\cong Q
$
and hence $C$ is left nilpotent.
\end{proof}

\begin{proposition}\label{prop:p-supp}
Suppose that $|Q|=p^a$ for a prime $p$. Then every Sylow
$p$-subskew brace $P$ of $X$ supplements $R$. More precisely,
$
X=R+P$ and $|P\cap R|=|R|_p.
$
In particular, $R$ has a left-nilpotent supplement.
\end{proposition}

\begin{proof}
Existence in the generality used here follows from Truman's Sylow
theorem \cite{Truman} (see also \cite{CDDFT} for earlier Sylow theorems
under broad structural hypotheses). Write $|R|_p=p^b$ and let $P$ be a
Sylow $p$-subskew brace. Then
$
|P|=|X|_p=p^{a+b}.$
For the natural projection $\rho\colon X\to Q$, we have
$
\ker(\rho|_P)=P\cap R$ and
$|P\cap R|\leq p^b.
$ Therefore
\[
|\rho(P)|
 =\frac{|P|}{|P\cap R|}
 \geq\frac{p^{a+b}}{p^b}
 =p^a
 =|Q|.
\] Since $\rho(P)\leq Q$, we get $\rho(P)=Q$ and
$
|P\cap R|=p^b=|R|_p.
$ Since both underlying groups of $P$ are $p$-groups,
Theorem~\ref{thm:csv} shows that $P$ is left nilpotent.
\end{proof}

\begin{corollary}\label{prop:hall-supp}
Let $\sigma=\pi(Q)$, and let $H$ be a Hall
$\sigma$-subskew brace of $X$. Then
$
X=R+H$, $|H\cap R|=|R|_\sigma,$ $
H/(H\cap R)\cong Q,$ and $\RL(H)\leq H\cap R.
$
In particular, if both $(X,+)$ and $(X,\cdot)$ are solvable, such a
Hall $\sigma$-subskew brace exists.
\end{corollary}

\begin{proof}
Let $\rho:X\to Q$ be the natural projection. Since $H$ is a Hall
$\sigma$-subskew brace,
$
[Q:\rho(H)]=[X:R+H]
$
divides $[X:H]$ which is a $\sigma'$-number. On the other hand,
$[Q:\rho(H)]$ divides $|Q|$ which is a $\sigma$-number. Therefore
$
[Q:\rho(H)]=1
$, that is $\rho(H)=Q$ and hence
$
X=R+H.
$

Let $J=H\cap R$. Then 
$
H/J\cong Q
$
and
$
|H|
 =|X|_\sigma
 =|R|_\sigma|Q|
 =|J||Q|,
$
and consequently
$
|J|=|R|_\sigma.
$
Since $H/J\cong Q$ is left nilpotent, we have
$
\RL(H)\leq J=H\cap R.
$

If both $(X,+)$ and $(X,\cdot)$ are solvable, the existence of $H$ follows
from \cite{Truman}.
\end{proof}
\subsection{Failure of left-nilpotent supplements}

The preceding existence results do not hold without additional
hypotheses.

\begin{theorem}\label{thm:no-supp}
There exists a finite skew brace $X$ for which $\RL(X)$ has no proper
supplement. In particular, $\RL(X)$ has no left-nilpotent supplement.
The skew brace
$
X=\SB{18}{15}
$
is such an example.
\end{theorem}

\begin{proof}
For $X=\SB{18}{15}$, we have
$
(X,+)\cong(X,\cdot)\cong S_3\times C_3.
$
Its left series has sizes $18>3$, and $L_\infty(X)$ is an ideal of
order $3$. Hence
$
R=\RL(X)=L_\infty(X)
$
by Theorem~\ref{thm:residual}(4).

The multiset of orders of all subskew braces of $X$ is
$
1,\;2,\;2,\;2,\;3,\;3,\;3,\;3,\;6,\;9,\;18.
$
If $C$ supplements $R$, then
$
|R|\,|C|=|X|\,|C\cap R|.
$
Since $|R|=3$ and $|X|=18$, it follows that
$
|C|=6|C\cap R|.
$
If $|C\cap R|=3$, then $|C|=18$, and hence $C=X$. Therefore every
proper supplement must satisfy
$
|C|=6
\quad\text{and}\quad
C\cap R=0.
$
However, every subskew brace of order $6$ contains $R$. Hence no
proper supplement exists. Since $L_\infty(X)\neq0$, the skew brace
$X$ itself is not left nilpotent.
\end{proof}

\section*{Computational methods}

The computations used in Examples~\ref{ex:18} and~\ref{ex:36} and
in the proof of Theorem~\ref{thm:no-supp} were carried out using
GAP~4.14.0 and version~0.10.7 of the GAP package
\textsf{YangBaxter}~\cite{YB}. A GAP verification script is available
from the authors upon request.
\section*{Declaration of Generative AI and AI-assisted technologies
in the writing process}

During the preparation of this manuscript, the authors used AI assistants
(OpenAI's ChatGPT and Anthropic's Claude) for language editing,
improvements to the presentation, and discussions of possible proof
strategies. All mathematical arguments were checked and verified by the
authors.


{\footnotesize
\begin{thebibliography}{99}
\enlargethispage{8\baselineskip}

\bibitem{CDDFT}
A.~Caranti, I.~Del Corso, M.~di Matteo, M.~Ferrara and M.~Trombetti,
\emph{On the Sylow theorem for skew braces},
Canad. Math. Bull. (2026), 1--15, published online 3 August 2026.
doi:10.4153/S0008439526102392; arXiv:2506.00940.

\bibitem{CedoVendraminBook}
F.~Ced\'o and L.~Vendramin,
\emph{Groups, Radical Rings, and the Yang--Baxter Equation:
A Combinatorial Approach to Solutions},
Progress in Mathematics, vol.~361, Birkh\"auser, Cham, 2026.

\bibitem{CSV}
F.~Ced\'o, A.~Smoktunowicz and L.~Vendramin,
\emph{Skew left braces of nilpotent type},
Proc. Lond. Math. Soc. (3) \textbf{118} (2019), no.~6, 1367--1392.

\bibitem{DameleSZ}
M.~Damele,
\emph{A Schur--Zassenhaus theorem for finite skew braces},
preprint, arXiv:2606.29295, 2026.

\bibitem{DE}
M.~Damele and G.~Ercan,
\emph{Ideals and solvability in skew braces},
preprint, arXiv:2607.19955, 2026.

\bibitem{DelCorso}
I.~Del Corso,
\emph{Module braces: relations between the additive and the multiplicative groups},
Ann. Mat. Pura Appl. (4) \textbf{202} (2023), 3005--3025.
doi:10.1007/s10231-023-01349-4.
\bibitem{EGGK}
G.~Ercan, \c{S}.~G\"ul, \.{I}.~\c{S}.~G\"ulo\u{g}lu and
M.~Y.~K{\i}zmaz,
\emph{Sylow theory and the nilpotency class of left nilpotent skew
braces},
preprint, arXiv:2606.25691v2, 2026.
\bibitem{GV}
L.~Guarnieri and L.~Vendramin,
\emph{Skew braces and the Yang--Baxter equation},
Math. Comp. \textbf{86} (2017), no.~307, 2519--2534.

\bibitem{IM}
I.~M. Isaacs and U.~Meierfrankenfeld,
\emph{Repeated and final commutators in group actions},
Proc. Amer. Math. Soc. \textbf{140} (2012), no.~11, 3777--3783.

\bibitem{Isaac}
I.~M. Isaacs,
\emph{Finite Group Theory},
Graduate Studies in Mathematics, vol.~92, American Mathematical Society,
Providence, RI, 2008.

\bibitem{KSV}
A.~Konovalov, A.~Smoktunowicz and L.~Vendramin,
\emph{On skew braces and their ideals},
Exp. Math. \textbf{30} (2021), no.~1, 95--104.

\bibitem{KSVErratum}
A.~Konovalov, A.~Smoktunowicz and L.~Vendramin,
\emph{Erratum to the paper ``On skew braces and their ideals''},
Exp. Math. \textbf{31} (2022), no.~1, 346.

\bibitem{Rump}
W.~Rump,
\emph{Braces, radical rings, and the quantum Yang--Baxter equation},
J. Algebra \textbf{307} (2007), no.~1, 153--170.

\bibitem{Truman}
P.~J. Truman,
\emph{Analogues of Sylow's first theorem, Cauchy's theorem, and Hall's theorem for skew braces},
preprint, arXiv:2606.18414, 2026.

\bibitem{YB}
L.~Vendramin and A.~Konovalov,
\emph{YangBaxter}, GAP package, version 0.10.7, 2025.

\end{thebibliography}
}
\end{document}